\documentclass{amsart}

\usepackage{amsmath,amstext,amssymb,mathrsfs,amscd,amsthm}
\usepackage{amsfonts}
\usepackage[all]{xy}
\usepackage{booktabs}
\usepackage{verbatim}
\usepackage{enumitem}
\usepackage{color}

\usepackage{graphicx}                   
\usepackage{subfig}
\usepackage{xspace}

\newtheorem{lemma}{Lemma}
\newtheorem{proposition}[lemma]{Proposition}
\newtheorem*{theorem}{Theorem}

\newtheorem{corollary}[lemma]{Corollary}

\newtheorem{predf}[lemma]{Definition} 

\newtheorem{preremark}[lemma]{Remark}  
\newenvironment{remark}{\begin{preremark}\rm}{\end{preremark}}
\newtheorem{preremark0}[lemma]{Remark}  

\newtheorem*{prenotation}{Notation}

\numberwithin{equation}{section}

\newcommand{\bZ}{\mathbb{Z}}

\newcommand\lra{\longrightarrow}

\newcommand\hocolim{\operatorname*{hocolim}}

\newcommand{\X}{\mathbf{X}}
\newcommand{\K}{\mathcal{K}}

\def\Id{\mathrm{Id}}

\newcommand{\cB}{\mathcal{B}}

\newcommand{\cF}{\mathcal{F}}

\newcommand{\cS}{\mathcal{S}}

\newcommand{\cX}{\mathcal{X}}

\usepackage{mathrsfs}

\newcommand\mnote[1]{}

\def\Tot{\operatorname{Tot}}
\def\rank{\operatorname{rank}}
\def\Complex{\mathbf{SComp}}
\def\Poset{\mathbf{Poset}}
\def\K{\mathscr{K}}
\def\X{\mathscr{X}}
\def\Setp{\mathbf{Set}_\bullet}
\def\Set{\mathbf{Set}}
\def\Topp{\mathbf{Top}_\bullet}
\def\Top{\mathbf{Top}}
\def\boldast{\mathbf{\ast}}
\def\Kh{Kh}
\def\jmin{j_{\mathrm{min}}}

\def\Fmin{\cF^{\jmin}}

\def\Fmintop{\tilde{\cF}^{\jmin}}
\def\Fset{\hat{\cF}^{\jmin}}
\def\Gmin{\Fmintop_+}

\newcommand{\cube}[1]{\mathbf{2^{#1}}}
\newcommand{\cubemas}[1]{\mathbf{2^{#1}_+}}
\def\cero{\vec{0}}
\def\uno{\vec{1}}

\title{Extreme Khovanov spectra}
\author{Federico Cantero Mor\'an and Marithania Silvero}
\thanks{Both authors were supported by project MTM2016-76453-C2 (AEI/FEDER, UE) and acknowledge financial support from the Spanish Ministry of Economy and Competitiveness through the Mar\'ia de Maeztu Programme for Units of Excellence in R\&D (MDM-2014-0445).}

\begin{document}
\begin{abstract}
We prove that the spectrum constructed by Gonz\'alez-Meneses, Manch\'on and the second author is stably homotopy equivalent to the Khovanov spectrum of Lipshitz and Sarkar at its extreme quantum grading.
\end{abstract}
\maketitle

\vspace{-0.2cm}

\section{Introduction}

Khovanov homology is a powerful link invariant introduced by Mikhail Khovanov in \cite{Khovanov00} as a categorification of the Jones polynomial. More precisely, given an oriented diagram $D$ representing a link $L$, he constructed a finite $\bZ$-graded family of chain complexes
\[\xymatrix{\ldots\ar[r]& C^{i,j}(D)\ar[r]^-{d_i} & C^{i+1,j}(D)\ar[r]^-{d_{i+1}} & C^{i+2,j}(D)\ar[r]&\ldots}\]
whose bigraded homology groups, $\Kh^{i,j}(D)$, are link invariants satisfying 
$$ J(L) (q) = \sum_{ij} q^j(-1)^i \rank(Kh^{i,j}(L)),
$$
where $J(L)$ is the Jones polynomial of $L$. The groups $\Kh^{i,j}(L)$ are known as \textit{Khovanov homology groups} of $L$, and the indexes $i$ and $j$ as \emph{homological} and \emph{quantum gradings}, respectively. 

A decade later, Lipshitz and Sarkar \cite{LSKhovanov} constructed a $\bZ$-graded family of spectra $\cX^{j}(D)$ associated to a link diagram $D$, and they proved that
\begin{quote} \emph{For each $j\in \bZ$, the spectrum $\cX^j(D)$ is a link invariant up to homotopy and there is a canonical isomorphism $H^*(\cX^j(D))\cong \Kh^{*,j}(D)$.}
\end{quote}
The construction of these spectra was later simplified in \cite{LSS} and \cite{LSS2}, where it was shown that each spectrum $\cX^j$ can be obtained as the suspension spectrum of the realisation of a certain cubical functor on pointed topological spaces.

For a given link diagram $D$, the Khovanov chain complex is trivial for all but finitely many $j$'s. Let $\jmin(D)$ 
 be the minimal 
 quantum grading such that the complex $\{C^{i,j}(D), d_i\}$ is non-trivial. In \cite{GMS} Gonz\'alez-Meneses, Manch\'on and Silvero introduced a simplicial complex $X_D$ satisfying the following
\begin{quote} \emph{The simplicial complex $X_D$, if not contractible, is a link invariant up to stable homotopy and there is a canonical isomorphism $H^{*+n_--1}(X_D)\cong \Kh^{*,\jmin}(D)$, with $n_-$ the number of negative crossings of $D$.}
\end{quote}
In this paper we show that, for the minimal quantum grading, both constructions are stably homotopy equivalent.

\section{A stable homotopy equivalence}

\subsection{States and enhancements} Let $\cube{n}$ be the poset $\{1\to 0\}^n$, which has an initial element $\uno = (1,1,\ldots,1)$ and a terminal element $\cero=(0,0,\ldots,0)$, and write $|v| = \sum_{i=1}^{n} v_i$.

Let $D$ be an oriented link diagram with $n$ ordered crossings, where $n_+$ ($n_-$) of them are positive (negative). A state is an assignation of a label, $0$ or $1$, to each crossing in $D$. There is a bijection between the set $\mathcal{S}$ of states of $D$ and the elements of $\cube{n}$ by considering $v \in \cube{n}$ as the state that assigns the $i^{th}$ coordinate of $v$ to the $i^{th}$ crossing of $D$. Write $D(v)$ for the set of topological circles and chords obtained after smoothing each crossing of $D$ according to its label by following Figure \ref{Labels}.

\begin{figure}
\centering
\includegraphics[width = 8.5cm]{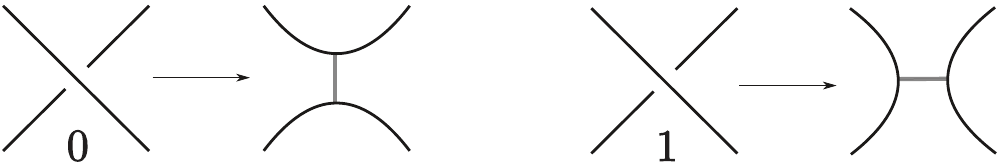}
\caption{\small{The smoothing of a crossing according to its $0$ or $1$ label.}} 
\label{Labels}
\end{figure}

An enhacement of a state $v$ is a map $x$ assigning a sign $\pm 1$ to each of the $|D(v)|$ circles in $D(v)$. Write $\tau(v,x) = \sum x(c)$ where $c$ ranges over all circles in $D(v)$, and define, for the enhanced state $(v,x)$, the integers
$$h(v,x) = h(v) = -n_- + |v|, \quad \quad q(v,x) = n_+ - 2n_- + |v| + \tau(v,x). $$

Let $j_{\min} = \min\{{q(v,x) \mid (v,x) \mbox{ is an enhanced state of } D}\}$, and for any state $v$ write $v^- = (v,x_-)$ with $x_-$ the constant function with value $-1$.
\begin{proposition}\cite[Proposition 4.1]{GMS}\label{prop:1_Lenguaje_Nudos}
In this setting, $j_{\min} = q(\cero^-)$ and $q(v,x) = j_{\min}$ if and only if $(v,x) \in \mathcal{S}_{\min}$, where
$$\mathcal{S}_{\min} =  \{\text{enhanced states $(v,x)$ such that $|D(v)| = |D(\cero)| + |v|$ and $x=x_-$}\}.$$
In particular, $j_{\min} = n_+ - 2n_- - |D(\cero)|$.
\end{proposition}
%
\begin{proposition}\label{prop:2_Lenguaje_Nudos}  Let $v\in \cS_{\min}$. If $u< v$ then $u^-\in \cS_{\min}$.
\end{proposition}
\begin{proof} 
Note that the Khovanov differential $d$ either splits one circle into two or merges two circles into one. As $|D(v)| = |D(\cero)| + |v|$, necessarily $v$ is obtained from $\cero$ by performing $|v|$ splittings in the crossings corresponding to non-zero coordinates of $v$. Hence, if $u$ and $v$ differ on $k$ coordinates, $v$ is obtained from $u$ by performing $k$ splittings, that is, $|D(u)| = |D(v)|-k = |D(\cero)| + |u|$.
\end{proof}

\subsection{The simplicial complex for extreme Khovanov homology} Let $D$ be an oriented link diagram. In \cite{GMS}, a simplicial complex $X_D$ was constructed, whose simplicial cochain complex is canonically isomorphic to the extreme Khovanov complex $\{C^{i,\jmin}(D),d_i\}$ shifted by $n_--1$. Next, we review the construction of $X_D$ (cf.\ Figure \ref{Hexagon}).
  
The Lando graph $G_D$ associated to $D$ consists of a vertex for each chord in $D(\cero)$ having both endpoints in the same circle, and an edge between two vertices if the endpoints of the corresponding chords alternate in the same circle. The simplicial complex $X_D$ is defined as the independence complex of the graph $G_D$; in other words, the simplices of $X_D$ are the subsets of pairwise non-adjacent vertices of $G_D$. Alternatively, it is the clique complex of the complement graph of $G_D$. 

\begin{figure}
\centering
\includegraphics[width = 12.1cm]{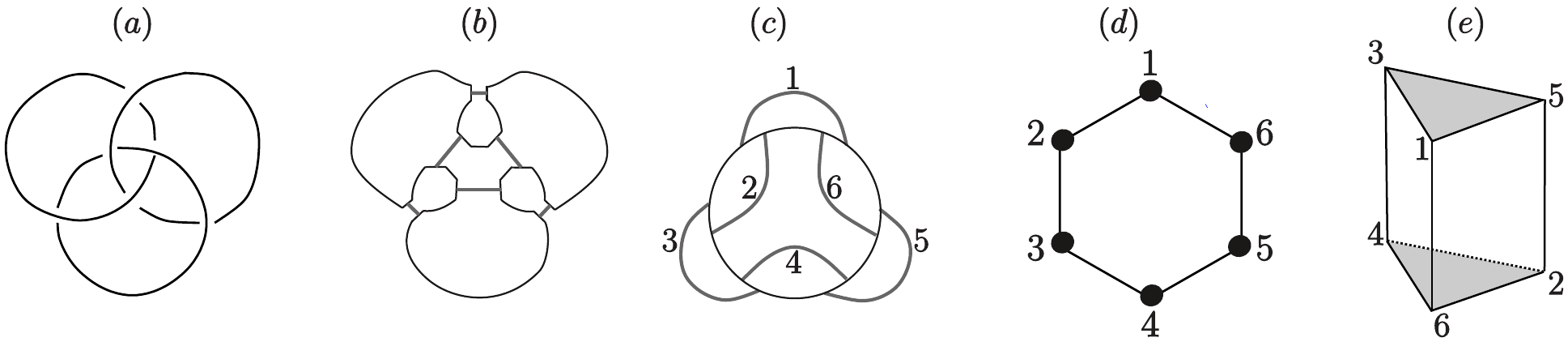}
\caption{\small{$(a)$ A link diagram $D$; $(b)$ and $(c)$ are two versions of $D(\cero)$; $(d)$ is the associated Lando graph $G_D$; $(e)$ is the geometric realisation of the simplicial complex $X_D$.}} 
\label{Hexagon}
\end{figure}



\subsection{Functors to the Burnside category} Let $\Topp$ be the category of pointed topological spaces with basepoint $\boldast$. Let $\Setp$ be the category of pointed sets, which we wiew as a subcategory of $\Topp$ as the subcategory of discrete spaces. Let $\cB$ be the Burnside $2$-category for the trivial group, whose objects are finite sets, morphisms are spans and $2$-morphisms are correspondences. We will freely refer to the results and notation of \cite{LSS} and \cite{LSS2} in what follows (see also the survey \cite{LS17}).

The category $\Setp$ sits inside $\cB$ by sending a pointed set $A$ to $A \smallsetminus \{\boldast\}$, and a morphism $f\colon A\to B$ to the span 
\[A\smallsetminus \{\boldast\} \hookleftarrow A\smallsetminus f^{-1}(\boldast)\overset{f}{\lra} B\smallsetminus \{\boldast\}.\]

Recall from \cite[Definition~5.1]{LSS} that an \emph{$N$-dimensional spatial refinement} of a functor $F\colon \cube{n}\to \cB$ is another functor $\tilde{F}\colon \cube{n}\to \Topp$ with values in wedges of spheres of dimension $N$ satisfying certain properties. The following observation is straightforward from that definition:
\begin{lemma} A functor $F\colon \cube{n}\to \cB$ has a $0$-dimensional spatial refinement $\tilde{F}$ if and only if $F$
factors as $\cube{n}\to \Setp\hookrightarrow \cB$. If this is the case, the refinement is $\tilde{F}\colon \cube{n}\to \Setp\subset \Topp$.
\end{lemma}


Let $\cubemas{n}$ be the poset obtained as follows: Take a second copy of $\cube{n}$, and rename its terminal object $\cero$ as $\circ$. The poset $\cubemas{n}$ is the union of both copies along the subposet $\cube{n}\smallsetminus \{\cero\}$. Alternatively, it is the result of adding two cones to $\cube{n}\smallsetminus \{\cero\}$ with apices $\cero$ and $\circ$. If $\tilde{F}\colon \cube{n}\to \Topp$ is an $N$-dimensional spatial refinement, then its totalisation is defined as follows: extend $\tilde{F}$ to a functor $\tilde{F}_+\colon \cubemas{n}\to \Topp$ by declaring $\tilde{F}_+(\circ)=\boldast$ and define:
\[\Tot \tilde{F} = \hocolim \tilde{F}_+\in \Topp.\]

\subsection{Khovanov spectra} Fix a link diagram $D$ and let $\cF\colon \cube{n}\to \cB$ be the functor constructed in \cite[Proposition 6.1]{LSS2} whose value at a vertex $v$ is the set of all possible enhancements associated to the state $v$. Let $\cF^j$ be the subfunctor whose values are those enhancements with quantum grading $j$. If $\tilde{\cF}^j$ is an $N$-dimensional spatial refinement of $\cF^j$, then the Khovanov spectrum of Lipshitz and Sarkar in quantum grading $j$ is \cite[Theorem~3]{LSS}
\begin{equation}
\tag{1}\label{eq:1}  \cX^{j}\simeq \Sigma^{-N-n_-}\Sigma^\infty\Tot \tilde{F}^j.
\end{equation}
When $j=\jmin$, we can restate Propositions \ref{prop:1_Lenguaje_Nudos} and \ref{prop:2_Lenguaje_Nudos} in the following way:

\begin{proposition}\label{prop:Fmin} The value of $\Fmin$ at a vertex $v\in \cube{n}$ is either the singleton $x_-$ for the case when $(v, x_-) \in \cS_{\min}$, or empty otherwise. Moreover, the value of $\Fmin$ at an arrow $v>u$ is, depending on the values of $\Fmin(u)$ and $\Fmin(v)$, \mnote{fc:ampliar tabla}

\[\begin{array}{|c|c|c|} \hline
 & \Fmin(v) = \emptyset & \Fmin (v)=x_-  \\ \hline
\Fmin(u) = \emptyset & \Id_\emptyset& \nexists \\ \hline
\Fmin(u)=x_-  & \emptyset \to \Fmin(u)& \Fmin(v)\cong \Fmin(u) \\ \hline
\end{array}\]
\end{proposition}
In particular, we obtain the following corollary:
\begin{corollary}\label{cor:spatial} $\Fmin$ factors through $\Setp$ and therefore the factorisation $\Fmintop$ is the $0$-dimensional spatial refinement of $\Fmin$. In fact, it further factors through the inclusion $\Set\subset \Setp$ sending a set $A$ to the pointed set $A\cup \{\boldast\}$. If we write $\Fset$ for the latter factorisation, we get
\[\xymatrix{
\cube{n} \ar[d]^{\Fset}\ar[drr]^{\Fmintop} \ar[rr]^{\Fmin} &&\cB\\
\Set \ar@{^{(}->}[rr] && \Setp \ar@{^{(}->}[u]
}\]
\end{corollary}
\subsection{A homotopy equivalence} Let $\Poset$ be the category of posets, and let $\Complex$ be the category of simplicial complexes. There are functors
\[\xymatrix{\Poset \ar@<.5ex>[r]^-{\K} & \ar[l]<.5ex>^-{\X} \Complex \ar[r]^-{|\, \cdot\, |} & \Top,}\]
where $\X$ takes a simplicial complex to its poset of non-empty faces, $\K$ takes a poset $P$ to the simplicial complex whose $0$-simplices are the elements of $P$, and whose $i$-simplices are ascending chains of $i+1$ elements in $P$. The functor $|\cdot|$ takes a simplicial complex to its realisation. The composition $\K\circ \X$ takes a simplicial complex $Y$ to its barycentric subdivision $\mathrm{sd}(Y)$. We will denote the composition $|K(\cdot)|$ by $\|\cdot\|$. If $P$ is a poset and $F\colon P\to \Top$ is a functor taking every element of $P$ to a singleton, then $\|P\|$ is a model for the homotopy colimit of $F$.

Let $\cS_{\min}'\subset \cube{n}$ be the subposet of those states $v$ such that $(v,x_-)\in \cS_{\min}$.  
The poset $\cube{n}$ can be identified with the poset of faces of the $(n-1)$-dimensional simplex with the arrows reversed, where we identify $\cero$ with the empty face and $\uno$ with the top-dimensional face. Under this identification, the poset of faces of $X_D$ becomes precisely $\cS_{\min}'$ \cite[Proposition 4.3]{GMS}. Therefore, if $F\colon \cS_{\min}'\to \Top$ is a functor with values on singletons, then
\begin{equation}\tag{2}\label{eq:2}\hocolim F\simeq \|\cS_{\min}'\smallsetminus \{\cero\}\| = |\mathrm{sd}(X_D)|\cong |X_D|.\end{equation}

\begin{theorem}\label{prop:spectra} There is a homotopy equivalence
\[\cX^{\jmin}\simeq \Sigma^{1-n_-}\Sigma^\infty |X_D|.\]
\end{theorem}
\begin{proof}
From \eqref{eq:1} and Corollary \ref{cor:spatial}, we have that $\cX^{\jmin}\simeq \Sigma^{-n_-}\Sigma^\infty\hocolim \Gmin$. We will prove that $\hocolim \Gmin\simeq \Sigma\|\cS_{\min}'\smallsetminus \{\cero\}\|$, and the result will follow from the homeomorphism $\|\cS_{\min}'\smallsetminus \{\cero\}\|\cong |X_D|$.

As $\cubemas{n}$ is constructed as the pushout of two cubes, there is a pushout diagram
\[\xymatrix{
\hocolim \Gmin|_{\cube{n}\smallsetminus \{\cero\}}\ar[r]\ar[d] & \hocolim \Gmin|_{\cube{n}}\ar[d] \\
\hocolim \Gmin|_{\cubemas{n}\smallsetminus \{\cero\}}\ar[r] & \hocolim \Gmin,
}\]
and as the two cubes have final elements $\cero$ and $\circ$, we have
\[\hocolim \Gmin|_{\cube{n}}\simeq \Gmin(\cero)=\{x_-,\boldast\},\quad \hocolim \Gmin|_{\cubemas{n}\smallsetminus \{\cero\}}\simeq \Gmin(\circ)=\boldast.\]
We now proceed to the computation of the upper left term in the diagram. Recall from the second part of Corollary \ref{cor:spatial} that $\Fmintop$ factors as $\Fset\colon \cube{n}\to \Set\subset \Setp$. Since the inclusion $\Top\subset \Topp$ is a left adjoint, it preserves colimits, and therefore
\[\hocolim \Gmin|_{\cube{n}\smallsetminus \{\cero\}} = \hocolim \Fset|_{\cube{n}\smallsetminus \{\cero\}}\cup \{\boldast\}\]
Now, from Proposition \ref{prop:Fmin}, it follows that $\Fset(u)$ is either a point or empty depending on whether $u$ belongs to $\cS_{\min}'$ or not; therefore 
\[\hocolim \Fset|_{\cube{n}\smallsetminus \{\cero\}} = \hocolim \Fset|_{\cS_{\min}'\smallsetminus \{\cero\}}\]
and since the latter functor is constant with values on singletons, \eqref{eq:2} leads to
\[\hocolim \Fset|_{\cS_{\min}'\smallsetminus \{\cero\}}\simeq \|\cS_{\min}'\smallsetminus\{\cero\}\|.\]
Finally, we face again the original pushout diagram in $\Topp$:
\[\xymatrix{
\|\cS_{\min}'\smallsetminus \{\cero\}\|\cup \{\boldast\} \ar[r]\ar[d] & \boldast\ar[d] \\
\{x_-,\boldast\}\ar[r] & \hocolim \Gmin,
}\]
where the left vertical map collapses $\|\cS_{\min}'\smallsetminus \{\cero\}\|$ to $\{x_-\}$. Replacing $\{x_-,\boldast\}$ by $\mathrm{Cone}(\|\cS_{\min}'\smallsetminus \{\cero\}\|)\cup \{\boldast\}$ and the $\boldast$ in the upper right corner by $\mathrm{Cone}(\|\cS_{\min}'\smallsetminus \{\cero\}\|)$ with basepoint the apex of the cone, we obtain a homotopy equivalent cofibrant pushout diagram, whose colimit is the (unreduced) suspension of $\|\cS_{\min}'\smallsetminus \{\cero\}\|$.
\end{proof}

\begin{remark} One can similarly define a maximal quantum grading $j_{\max}$ and define a simplicial complex $Y_D$ as the Alexander dual of $X_{D^*}$ where $D^*$ is the mirror image of $D$ (cf.\ \cite[Theorem 7.4]{PrzytyckiSilvero}). The fact that the Khovanov spectrum of a link diagram is the Spanier-Whitehead dual of the Khovanov spectrum of its mirror image, immediately implies that $\cX^{j_{\max}}\simeq \Sigma^{n_+-1}\Sigma^\infty Y_D$. 
\end{remark}

\bibliographystyle{amsalpha}
\bibliography{biblio-article}

\end{document}